\documentclass[11pt]{amsart}

\usepackage{amssymb}
\usepackage{mathtools}
\usepackage{color}

\newtheorem{theo}{Theorem} 

\newtheorem{lemma}{Lemma}[section]
\newtheorem{prop}[lemma]{Proposition}
\newtheorem{corol}[lemma]{Corollary}

\theoremstyle{remark}
\newtheorem{remark}[lemma]{Remark}

\theoremstyle{definition}

\newcommand{\smsp}{\kern-.3mm}
\newcommand{\tnorm}[1]{\left|\smsp\left|\smsp\left|#1\right|\smsp\right|\smsp\right|}

\newcommand{\lin}{\textsc{l}}

\newcommand{\RR}{\mathbb{R}}

\newcommand{\eps}{\varepsilon}

\newcommand{\BBB}{\mathcal{B}}

\newcommand{\PPP}{\mathcal{P}}

\newcommand{\vu}{\vec{u}}

\newcommand{\tf}{\widetilde{f}}

\newcommand{\ovt}{\overline{t}}
\newcommand{\ovq}{\overline{q}}

\newcommand{\tu}{\tilde{u}}
\newcommand{\tg}{\tilde{g}}

\newcommand{\tv}{\tilde{v}}

\newcommand{\loc}{\rm loc}
\newcommand{\hdot}{\dot{H}^1}

\newcommand{\HcrA}{\dot{H}^{s_p}}
\newcommand{\HcrB}{\dot{H}^{s_p-1}}
\newcommand{\Hcr}{\HcrA\times \HcrB}
\newcommand{\Hen}{\hdot\times L^2}

\newcommand{\EMPH}[1]{\medskip\noindent\textit{#1}.}

\DeclareMathOperator{\supp}{supp}

\numberwithin{equation}{section} 

\title[Scattering for super-critical wave]{Scattering for radial, bounded solutions of focusing supercritical wave equations}
\author[T.~Duyckaerts]{Thomas Duyckaerts$^1$}
\author[C.~Kenig]{Carlos Kenig$^2$}
\author[F.~Merle]{Frank Merle$^3$}
\thanks{$^1$LAGA, Paris 13 University (UMR 7539). Partially supported by ERC grant DISPEQ and ERC advanced grant
no. 291214, BLOWDISOL}
\thanks{$^2$University of Chicago. Partially supported by NSF Grant DMS-0968472}
\thanks{$^3$Cergy-Pontoise (UMR 8088), IHES, CNRS. Partially supported by ERC advanced grant
no. 291214, BLOWDISOL}

\date{\today}

\begin{document}
\begin{abstract}
In this paper, we consider the wave equation in space dimension $3$ with an energy-supercritical, focusing nonlinearity. We show that any radial solution of the equation which is bounded in the critical Sobolev space is globally defined and scatters to a linear solution. As a consequence, finite time blow-up solutions have critical Sobolev norm converging to infinity (along some sequence of times). The proof relies on the compactness/rigidity method, pointwise estimates on compact solutions obtained by the two last authors, and channels of energy arguments used by the authors in previous works on the energy-critical equation.
\end{abstract}

\maketitle

\tableofcontents

\section{Introduction}
In this paper, we consider the following wave equation on an interval $I$ ($0\in I$)
\begin{equation}
\label{CP}
\left\{ 
\begin{gathered}
\partial_t^2 u -\Delta u-|u|^{p-1}u=0,\quad (x,t)\in \RR^3\times I\\
u_{\restriction t=0}=u_0\in \HcrA,\quad \partial_t u_{\restriction t=0}=u_1\in \HcrB,
\end{gathered}
\right.
\end{equation}
where $u$ is real-valued, $p>5$ and 
\begin{equation}
s_p=\frac{3}{2}-\frac{2}{p-1}.  
\end{equation} 
The equation is locally well-posed in $\HcrA\times \HcrB$. For any initial data $(u_0,u_1)\in \Hcr$, there exists a unique solution $u$ of \eqref{CP} defined on a maximal interval of existence $(T_-(u),T_+(u))$ (for the precise definition of solution that we use see \cite[Definition 2.7]{KeMe11}). If $u$ is a solution, we will denote $\vu=(u,\partial_tu)$.

The main aim of this paper is to extend to the focusing case the results in \cite{KeMe11} which were obtained only in the defocusing case. We show that if a radial solution has the property that the $\Hcr$ norm remains bounded up to the maximal positive time of existence $T_+$, then $T_+$ is infinite and the solution scatters at plus infinity to a linear solution. We thus obtain Theorem 4.1, Corollary 4.2 and Corollary 4.3 of \cite{KeMe11} in the focusing case. Because of the ``concentration-compactness/rigidity theorem" method of the second and third authors (\cite{KeMe08}, \cite{KeMe11}) matters are reduced to establishing a rigidity theorem (Theorem \ref{T:rigidity}) for radial solutions verifying the ``compactness'' or ``nondispersive'' property. In the proof of such rigidity results, the pointwise decay estimates for radial solutions with the ``compactness'' property, established in \cite{KeMe11}, Theorem 3.11, both in the focusing and defocusing cases, are fundamental. The second fundamental ingredient in \cite{KeMe11}, in the defocusing case, for the proof of the rigidity theorem, was the virial identities (or alternatively, as in \cite{KillipVisan11}, \cite{KeMe12}, the Morawetz identity). These identities are not useful in the focusing, supercritical case, and a new method had to be devised. The method we employ here is the ``channel of energy'' method of \cite{DuKeMe12P}, which is not dependent on global integral identities such as the virial identities. The main idea is to show, through the ``channel of energy'' method, that a non-zero radial solution, with the ``compactness'' property must coincide with a constructed stationary solution which can be shown not to be in $\dot{H}^{s_p}$ (because of the work in \cite{JosephLundgren72}). This provides the desired contradiction. 

It is worth noting that as a consequence of our main result, Theorem \ref{T:main}, a radial finite time blow-up solution $u$ (which certainly can exist in the focusing case) of the focusing equation must satify
$$ \limsup_{t\to T_+} \|(u(t),\partial_tu(t))\|_{\Hcr}=+\infty.$$
This is in  stark contrast to the energy critical case, where radial type II blow-up solutions exist (see \cite{KrScTa09}). The point here is that in the energy critical case, our constructed stationary solution is the ground state $W(x)=\left(1+|x|^2/3\right)^{-1/2}$, which in this case clearly is in $\dot{H}^1=\dot{H}^{s_p}$. It is also worth pointing out that, simultaneously to this work, the second author's student Ruipeng Shen \cite{Shen12P} also used the ``channel of energy'' method of \cite{DuKeMe12P} and decay estimates for radial solutions with the ``compactness'' property, in the spirit of \cite{KeMe11}, in conjunction with intricate gain of regularity arguments, to show that analoguous results to those in this paper also hold in the subcritical case $3<p<5$. This showcases the very special role that the energy critical nonlinearity $p=5$ plays in the range $3<p<\infty$. 

Our main results are the Theorem \ref{T:main} below as well as the rigidity result, Theorem \ref{T:rigidity} in Section \ref{S:reduction}.

\begin{theo}
 \label{T:main}
Let $u$ be a radial solution of \eqref{CP} with maximal interval of existence $(T_-,T_+)$ such that 
\begin{equation}
 \label{bounded}
\sup_{t\in (0,T_+)} \|(u(t),\partial_tu(t))\|_{\Hcr}<\infty.
\end{equation} 
Then $u$ is globally defined for positive times and scatters for positive times to a linear solution.
\end{theo}
The analoguous result was proved in \cite{KeMe11} for the defocusing equation (see \cite{KillipVisan11} for the defocusing equation in  the nonradial case). 

Plan of the paper: Section \ref{S:reduction} is devoted to the reduction of the rigidity theorem, Theorem \ref{T:rigidity}, to two special cases, Proposition \ref{P:rigidityA} and Proposition \ref{P:rigidityB}. The main special case, Proposition \ref{P:rigidityA}, is treated in Section \ref{S:global}, using the ``channel of energy'' method. Section \ref{S:self-sim} deals with the second special case, Proposition \ref{P:rigidityB}, which concerns radial finite time self-similar compact blow-up. This case is dealt with using the corresponding method in \cite{KeMe08} (which is considerably simpler in the radial case).

\section{Reduction of rigidity theorem to two cases}
\label{S:reduction}
By the same arguments than in the defocusing case (see \cite[Section 4]{KeMe11}), the proof of Theorem \ref{T:main} reduces to the proof of the following rigidity theorem:
\begin{theo}
\label{T:rigidity}
 Let $u$ be a solution of \eqref{CP}. Assume that there exists a continuous function $\lambda:[0,T_+(u))\to (0,\infty)$ such that
\begin{equation}
\label{defK+}
K_+=\left\{\left(\frac{1}{\lambda(t)^{\frac{2}{p-1}}}u\left(\frac{\cdot}{\lambda(t)},t\right), \frac{1}{\lambda(t)^{\frac{2}{p-1}+1}}\partial_tu\left(\frac{\cdot}{\lambda(t)},t\right)\right),\; t\in [0,T_+(u))\right\}
\end{equation} 
has compact closure in $\Hcr$ and
\begin{equation}
\label{bound_below+}
 \inf_{t\in [0,T_+(u))} \lambda(t)>0.
\end{equation} 
Then $u=0$.
\end{theo}
In this section we explain how to deduce Theorem \ref{T:rigidity} from the two following propositions, proved in Sections \ref{S:global} and \ref{S:self-sim} respectively:
\begin{prop}
 \label{P:rigidityA}
Let $u$ be a radial solution of \eqref{CP}. Assume that there exists a continuous function $\lambda:(T_-(u),T_+(u))\to (0,+\infty)$, such that
\begin{equation}
 \label{defK}
K=\left\{\left(\frac{1}{\lambda(t)^{\frac{2}{p-1}}}u\left(\frac{\cdot}{\lambda(t)},t\right), \frac{1}{\lambda(t)^{\frac{2}{p-1}+1}}\partial_tu\left(\frac{\cdot}{\lambda(t)},t\right)\right),\; t\in (T_-(u),T_+(u))\right\}
\end{equation} 
has compact closure in $\Hcr$ and 
\begin{equation}
\label{bound_below}
 \inf_{t\in (T_-(u),T_+(u))} \lambda(t)>0.
\end{equation} 
Then $u=0$.
\end{prop}
\begin{prop}
\label{P:rigidityB}
 There is no radial solution $u$ of \eqref{CP} such that $T_+=T_+(u)<\infty$ and 
\begin{equation}
\label{defK+'}
K_+:=\left\{\left(\frac{1}{(T_+-t)^{\frac{2}{p-1}}}u\left(\frac{\cdot}{T_+-t},t\right), \frac{1}{(T_+-t)^{\frac{2}{p-1}+1}}\partial_tu\left(\frac{\cdot}{T_+-t},t\right)\right),\; t\in [0,T_+)\right\}
\end{equation} 
has compact closure in $\Hcr$.
\end{prop}
By the general arguments of Section 6 of \cite{KeMe11}, assuming Propositions \ref{P:rigidityA} and \ref{P:rigidityB}, the proof of Theorem \ref{T:rigidity} reduces to prove the following:
\begin{prop}
\label{P:rigidityC}
 Let $u$ be as in Theorem \ref{T:rigidity}. Then $T_+(u)=\infty$. 
\end{prop}
The rest of this section is devoted to the proof of Proposition \ref{P:rigidityC}, assuming Propositions \ref{P:rigidityA} and \ref{P:rigidityB}. 

If $R>0$, we denote by $B_R=\{x\in \RR^3,\text{ s.t. } |x|\leq R\}$. Assume to fix ideas that $T_+(u)=1$. By a standard argument (see Lemma 4.14 and 4.15 of \cite{KeMe11}), there exists a constant $c>0$ such that 
\begin{equation}
\label{eq1}
\forall t\in [0,1),\quad (1-t)\lambda(t)\geq c 
\end{equation} 
and
\begin{equation}
\label{support}
\forall t\in [0,1), \quad \supp \vec{u}(t)\subset B_{1-t}.
\end{equation} 
We first prove the following:
\begin{lemma}
\label{L:interm1}
 Let $u$ be as in Theorem \ref{T:rigidity} with $T_+(u)<\infty$. Then there exists $\eps_0>0$ such that
\begin{equation}
\label{conclu_interm1}
\forall (t,t')\in [0,1)^2,\quad t<t'\Longrightarrow \eps_0\lambda(t) \leq \lambda(t')
\end{equation} 
\end{lemma}
\begin{proof}
\EMPH{Step 1: convergence for a well-chosen sequence of times}
We argue by contradiction. Assume that \eqref{conclu_interm1} does not hold. Then there exist sequences $\{t_n\}_n$, $\{\ovt_n'\}_n$ such that 
\begin{equation}
 \label{eq3}
\forall n,\quad \lambda(\ovt'_n)\leq \frac{1}{n}\lambda(t_n),\quad t_n<\ovt'_n.
\end{equation} 
Note that $\lim_n \ovt_n'=1$. Otherwise we would have (after extraction) $\lim \ovt_n'=t_*\in [0,1)$ with $\lambda(t_*)=0$, a contradiction. 

Fix an index $n$. By \eqref{eq1} and the continuity of $\lambda$, there exists $t_n'\in [t_n,1)$ such that
\begin{equation}
 \label{eq4}
\lambda(t_n')=\min_{t_n\leq t<1} \lambda(t).
\end{equation} 
By \eqref{eq3}, $t_n<t_n'$ and
\begin{equation}
\label{eq5} 
\lambda(t_n')\leq \frac{1}{n}\lambda(t_n).
\end{equation}
By the intermediate value theorem, \eqref{eq1} and \eqref{eq5}, for all $n$ ,there exists $t_n''\in (t_n',1)$ such that
\begin{equation}
 \label{eq6}
\lambda(t_n'')=\lambda(t_n).
\end{equation} 
Define
\begin{equation}
 \label{eq7}
v_n(y,\tau)= \frac{1}{\lambda(t_n')^{\frac{2}{p-1}}}u\left(\frac{y}{\lambda(t_n')},t_n'+\frac{\tau}{\lambda(t_n')}\right),
\end{equation} 
and
\begin{equation*}
 \lambda_n(\tau)= \frac{\lambda\left( t_n'+\frac{\tau}{\lambda(t_n')}\right)}{\lambda\left(t_n'\right)}.
\end{equation*} 
Since $\overline{K}_+$ is compact in $\Hcr$ and $\left(v_n(0),\partial_{\tau}v_n(0)\right)\in K$ for all $n$, there exists $(v_0,v_1)\in \Hcr$ such that (after extraction of a subsequence)
$$ \lim_{n\to \infty} \left\| \left(v_n(0),\partial_{\tau}v_n(0)\right)-(v_0,v_1)\right\|_{\Hcr}=0.$$
Let $v$ be the solution of \eqref{CP} with initial data $(v_0,v_1)$ and $(\tau_-,\tau_+)$ its maximal interval of existence. We will show in the next step that $v$ satisfies the assumptions of Proposition \ref{P:rigidityA}. This implies, by Proposition \ref{P:rigidityA} $(v_0,v_1)=0$, and thus 
$$\lim_{n\to\infty}\|\vec{u}_n(t_n')\|_{\Hcr}=0,$$
which shows by the small data Cauchy theory for \eqref{CP} that $u=0$, contradicting our assumptions.

\EMPH{Step 2: construction of the scaling parameter and conclusion} 
Let
$$ \tau_n^-=(t_n-t_n')\lambda(t_n'),\quad \tau_n^+=(t_n''-t_n')\lambda(t_n').$$
By the definition \eqref{eq4} of $t_n'$, 
\begin{equation}
 \label{eq8}
\forall \tau \in [\tau_n^-,\tau_n^+],\quad \lambda_n(\tau)\geq 1.
\end{equation} 
Furthermore by \eqref{eq5},
\begin{equation}
 \label{eq9}
\lambda_n\left(\tau_n^{\pm}\right)=\frac{\lambda(t_n)}{\lambda(t_n')}\geq n.
\end{equation}

Let $\tau\in (\tau_-,\tau_+)$, and $\{\tau_n\}_n$ a sequence in $[\tau_n^-,\tau_n^+]$ 
such that $\lim_{n\to \infty}\tau_n=\tau$. Let us show that there exists $\mu\in(0,\infty)$ such that
\begin{equation}
\label{lim_mu}
 \lim_{n\to\infty}\lambda_n(\tau_n)=\mu.
\end{equation} 
Indeed, by standard perturbation theory for \eqref{CP} (see Theorem 2.11 of \cite{KeMe11}), 
\begin{equation}
\label{eq10}
\lim_{n\to\infty}
\vec{v}_n(\tau_n)=\vec{v}(\tau)\text{ in }\Hcr.
\end{equation} 
Using 
$$ \frac{1}{\lambda_n(\tau_n)^{\frac{2}{p-1}}}v_n\left(\frac{y}{\lambda_n(\tau_n)},\tau_n\right)=\frac{1}{\lambda\left(t_n'+\frac{\tau_n}{\lambda(t_n')}\right)^{\frac{2}{p-1}}} u\left(\frac{y}{\lambda\left(t_n'+\frac{\tau_n}{\lambda(t_n')}\right)},t_n'+\frac{\tau_n}{\lambda(t_n')}\right),
$$
and the similar equality for the time derivatives, we see that
\begin{equation}
\label{vninK}
\left(\frac{1}{\lambda_n(\tau_n)^{\frac{2}{p-1}}}v_n\left(\frac{y}{\lambda_n(\tau_n)},\tau_n\right),\frac{1}{\lambda_n(\tau_n)^{1+\frac{2}{p-1}}}\partial_{\tau}v_n\left(\frac{y}{\lambda_n(\tau_n)},\tau_n\right)\right)\in K 
\end{equation} 
for all $n$. As a consquence of \eqref{eq10} we deduce (after extraction of subsequences) that there exists $\mu>0$ such that \eqref{lim_mu} holds.

We next show by contradiction
\begin{equation}
 \label{eq11}
\liminf_n \tau_n^+\geq \tau_+,\quad \limsup_n\tau_n^-\leq \tau_-.
\end{equation} 
Indeed, if for example $\liminf_n \tau_n^+<\tau_+$, there exists a subsequence of $\{\tau^+_n\}_n$, still denoted by $\{\tau^+_n\}_n$, such that $\lim_n \tau_n^+=\tau\in [0,\tau_+)$. Thus by the preceding paragraph, the sequence $\{\lambda_n(\tau_n^+)\}_n$ converges to a limit in $(0,\infty)$ after extraction of a subsequence, contradicting \eqref{eq9}. 

Let $\tau\in (\tau_-,\tau_+)$, and define (after extraction in $n$),
$$\mu(\tau)=\lim_{n\to \infty}\lambda_n(\tau).$$
Passing to the limit in \eqref{vninK} (with $\tau_n=\tau$), we get
$$
\left(\frac{1}{\mu(\tau)^{\frac{2}{p-1}}}v\left(\frac{\cdot}{\mu(\tau)},\tau\right), \frac{1}{\mu(\tau)^{\frac{2}{p-1}+1}}\partial_{\tau}v\left(\frac{\cdot}{\mu(\tau)},\tau\right)\right)\in \overline{K},$$
and by \eqref{eq8} and \eqref{eq11}, $\mu(\tau)\geq 1$ for all $\tau \in (\tau_-,\tau_+)$. This shows as announced that $v$ satisfies the assumptions of Proposition \ref{P:rigidityA}, concluding the proof.
\end{proof}
We are now ready to proceed to the proof of Proposition \ref{P:rigidityC}. Define, for $t\in [0,1)$,
\begin{equation}
 \label{eq12}
\lambda_1(t)=\sup_{\tau\in [0,t]} \lambda(\tau),
\end{equation} 
and note that $\lambda_1(t)$ is a nondecreasing, continuous function of $t$. By Lemma \ref{L:interm1},
\begin{equation}
 \label{eq12'}
\forall t\in [0,1),\quad \lambda(t)\leq \lambda_1(t)\leq \frac{1}{\eps_0}\lambda(t)
\end{equation} 
and as a consequence, the set $K_{1+}$ defined as $K_+$ in \eqref{defK+}, but with $\lambda_1(t)$ instead of $\lambda(t)$ has compact closure in $\Hcr$. This implies that \eqref{eq1} is still valid with $\lambda$ instead of $\lambda_1$: there exists $c_1>0$ such that
\begin{equation}
 \label{eq13}
\forall t\in [0,1),\quad  (1-t)\lambda_1(t)\geq c_1.
\end{equation}  
For any large integer $n$, we let $t_n\in [0,1)$ be such that 
$$\lambda_1(t_n)=2^n.$$
Note that $\{t_n\}_n$ is increasing and converges to $1$ as $n$ tends to infinity. We distinguish two cases.

\EMPH{Case 1}
Assume:
$$\exists C>0\quad\text{s.t.}\quad\forall n,\quad t_{n+1}-t_n\leq \frac{C}{2^n}=\frac{C}{\lambda_1(t_n)}.$$
Let $n_0$ be a large integer. Then
$$1-t_{n_0}=\sum_{n\geq n_0} (t_{n+1}-t_n)\leq \frac{C}{2^{n_0-1}}=\frac{C}{\lambda_1(t_{n_0})}.$$
Thus 
$$ \lambda_1(t_{n_0})\leq \frac{C}{1-t_{n_0}}$$
for large $n_0$. Noting that $\lambda_1(t_{n_0})\leq \lambda_1(t)\leq 2\lambda_1(t_{n_0})=\lambda_1(t_{n_0+1})$ if $t_{n_0}\leq t\leq t_{n_0+1}$, we deduce (taking a large constant $C$)  
$$ \forall t\in [0,1),\quad \lambda_1(t)\leq \frac{C}{1-t}.$$
Combining with \eqref{eq13} we see that $u$ satisfies the assumptions of Proposition \ref{P:rigidityB}, a contradiction.

\EMPH{Case 2} We assume that we are not in case $1$, i.e. that there exists an increasing sequence of integers $\varphi(n)\to +\infty$ such that
\begin{equation}
 \label{eq15}
\forall n,\quad t_{\varphi(n)+1}-t_{\varphi(n)}\geq \frac{n}{2^{\varphi(n)}}=\frac{n}{\lambda_1(t_{\varphi(n)})}
\end{equation} 
Let 
$$ t_n'=\frac{t_{\varphi(n)}+t_{\varphi(n)+1}}{2},\quad t_{\varphi(n)}<t_n'<t_{\varphi(n)+1}.$$
Then by \eqref{eq15}, and since $\lambda_1$ is nondecreasing,
\begin{equation}
 \label{eq17}
t_{\varphi(n)+1}-t_n'\geq \frac{n}{2\lambda_1(t'_n)},\quad t_n'-t_{\varphi(n)}\geq \frac{n}{2\lambda_1(t'_n)}.
\end{equation} 
Define
\begin{equation}
 \label{eq18}
v_n(y,\tau)= \frac{1}{\lambda_1(t_n')^{\frac{2}{p-1}}}u\left(\frac{y}{\lambda_1(t_n')},t_n'+\frac{\tau}{\lambda_1(t_n')}\right).
\end{equation} 
Since $\left(v_n(0),\partial_{\tau}v_n(0)\right)\in K_{1+}$, there exists (after extraction of a subsequence) an element $(v_0,v_1)$ of $\Hcr$ such that
$$\lim_{n\to\infty} \left\|\left(v_n(0),\partial_{\tau}v_n(0)\right)-(v_0,v_1)\right\|_{\Hcr}=0.$$
Let 
$$ \tau_n^-=(t_{\varphi(n)}-t_n')\lambda_1(t_n'),\quad \tau_{n}^+=(t_{\varphi(n)+1}-t_n')\lambda_1(t_n'),$$
and note that by \eqref{eq17},
$$ \lim_{n\to\infty} \tau_n^{\pm}=\pm\infty.$$
Let also
$$\mu_n(\tau)=\frac{\lambda_1\left(t_n'+\frac{\tau}{\lambda_1(t_n')}\right)}{\lambda_1(t_n')},\quad \tau_n^-\leq \tau\leq \tau_n^+.$$
For $\tau\in \left[\tau_n^-,\tau_n^+\right]$, we have
\begin{equation}
 \label{eq19}
\left(\frac{1}{\mu_n(\tau)^{\frac{2}{p-1}}}v_n\left(\frac{\cdot}{\mu_n(\tau)},\tau\right),\frac{1}{\mu_n(\tau)^{1+\frac{2}{p-1}}}\partial_{\tau}v_n\left(\frac{\cdot}{\mu_n(\tau)},\tau\right)\right)\in K_{1+}.
\end{equation} 
Furthermore, by the definition of $t_n$,
\begin{equation}
 \label{eq20}
\forall \tau\in [\tau_n^-,\tau_n^+],\quad \frac{1}{2}\leq \mu_n(\tau)\leq 2.
\end{equation} 
As in the proof of Lemma \ref{L:interm1}, we deduce from \eqref{eq19}, \eqref{eq20} that for all $\tau$ in the interval of definition of $v$, there exists $\mu(\tau)\in [1/2,2]$ such that
$$ \left(\frac{1}{\mu(\tau)^{\frac{2}{p-1}}}v\left(\frac{\cdot}{\mu(\tau)},\tau\right),\frac{1}{\mu(\tau)^{1+\frac{2}{p-1}}}\partial_{\tau}v\left(\frac{\cdot}{\mu(\tau)},\tau\right)\right)\in \overline{K}_{1+}.
$$
Thus $v$ satisfies the assumptions of Proposition \ref{P:rigidityA}, which implies $v=0$, a contradiction with our assumption that $T_+(u)=1$. The proof of Proposition \ref{P:rigidityC} is complete.
\qed

\section{Main rigidity result}
\label{S:global}

In this section we prove Proposition \ref{P:rigidityA}. The outline is as follows: Subsection \ref{SS:preliminaries} is devoted to preliminaries on the linear wave equation (\S \ref{SSS:linear}), construction of a singular stationary solution to \eqref{CP} (\S \ref{SSS:singular}), a Cauchy problem related to \eqref{CP} (\S \ref{SSS:Cauchy}) and previous results of the two last authors \cite{KeMe11} on solutions of \eqref{CP} satisfying the assumptions of Proposition \ref{P:rigidityA} (\S \ref{SSS:pointwise}).
The core of the proof is in Subsection \ref{SS:proof}. The main idea is to show, through the channel of energy method of \cite{DuKeMe12P} that a nonzero solution which satisfies the compactness property must coincide with the stationary solution constructed in \S \ref{SSS:singular}. Since this solution is not in $\dot{H}^{s_p}$ we will obtain a contradiction. 

\subsection{Preliminaries}
\label{SS:preliminaries}
\subsubsection{Energy channels for linear waves}
\label{SSS:linear}
We recall from \cite{DuKeMe11a}:
\begin{lemma}
\label{L:linear}
 Let $u$ be a radial solution of 
\begin{equation}
 \label{linear_CP}
\left\{
\begin{gathered}
 \partial_t^2u-\Delta u=0,\quad x\in \RR^3,\;t\in \RR\\
\vec{u}_{\restriction t=0}=(u_0,u_1)\in \hdot\times L^2.
\end{gathered}
\right.
\end{equation} 
Let $r_0\geq 0$. Then for all $t\geq 0$ or for all $t\leq 0$:
\begin{equation*}
 \int_{r_0+|t|}^{+\infty} \left(\partial_r (ru(r,t))\right)^2+\left(\partial_t (ru(r,t))\right)^2\,dr
\geq \frac 12 \int_{r_0}^{+\infty} \left(\partial_r (ru_0(r))\right)^2+\left(ru_1(r)\right)^2\,dr
\end{equation*} 
\end{lemma}
\subsubsection{A singular stationary solution}
\label{SSS:singular}
\begin{prop}
\label{P:singular}
Let $p>5$, and $\ell\in \RR\setminus \{0\}$. Then there exists a radial, $C^2$ solution of
\begin{equation}
 \label{B3}
\Delta Z_{\ell}+|Z_{\ell}|^{p-1}Z_{\ell}=0 \text{ on }\RR^3\setminus \{0\}
\end{equation} 
such that
\begin{gather}
 \label{B4}
\forall r\geq 1,\quad \left|r\,Z_{\ell}(r)-\ell\right|\leq \frac{C}{r^2}\\
\label{B5}
\lim_{r\to\infty} r^2 \frac{d Z_{\ell}}{dr}=-\ell^2.
\end{gather} 
Furthermore, $Z_{\ell}\notin L^{q_p}$, where $q_p:=\frac{3(p-1)}{2}$ is the critical Sobolev exponent corresponding to $s_p$.
\end{prop}

\begin{proof}
 Once $Z_1$ is constructed, $Z_{\ell}=\pm \frac{1}{\lambda^{\frac{2}{p-1}}}Z_1\left(\frac{r}{\lambda}\right)$ satisfies the conclusions of Proposition \ref{P:singular} if $\lambda^{\frac{p-3}{p-1}}=|\ell|$ and $\pm$ is the sign of $\ell$. In the sequel we will assume $\ell=1$ and construct $Z_1$.

Let $f\in C^2\left(\RR^3\setminus\{0\}\right)$ be such that there exists $C>0$ with
\begin{equation}
 \label{B6} 
\left|rf(r)-1\right|\leq\frac{C}{r^2}\text{ for }r\geq 1\text{ and }\lim_{r\to\infty}r^2\frac{df}{dr}(r)=-1.
\end{equation} 
Let $g(r)=rf(r)$. Then the equation $\Delta f+|f|^{p-1}f=0$ on $\RR^3\setminus \{0\}$ is equivalent to
\begin{equation}
 \label{B6'}
\forall r>0, \quad g(r)=1-\int_{r}^{+\infty} \int_s^{+\infty} \frac{1}{\sigma^{p-1}}|g(\sigma)|^{p-1}g(\sigma)\,d\sigma\,ds.
\end{equation} 

\EMPH{Step 1: solution for large $r$}

We let $r_0>1$ be a large parameter to be specified later and
\begin{gather}
 \label{B7}
V:= \left\{ g\in C^1\left([r_0,+\infty),\RR\right)\text{ s.t. }\|g\|_{V}:=\sup_{r\geq r_0}\left(|g(r)|r^{p-4}+|g'(r)|r^{p-3}\right)<\infty\right\}\\
\label{B8}
B:=\left\{g\in V\text{ s.t. }\|g-1\|_V\leq 1\right\}.
\end{gather} 
Note that $V$ is a Banach space. 
For $g\in B$, we let
\begin{equation}
 \label{B9}
T(g):= 1-\int_r^{+\infty} \int_s^{+\infty} \frac{1}{\sigma^{p-1}}\left|g(\sigma)\right|^{p-1}g(\sigma)\,d\sigma\,ds.
\end{equation} 
We next check that $T:B\to B$ and is a contraction on $B$. 

Indeed, if $g\in B$, then $|g(r)|\leq 2$ for $r\geq r_0$ and thus, $T(g)\in C^1([r_0,\infty))$ and
\begin{equation}
 \label{B10}
\left|T(g)-1\right|\leq \frac{2^p}{(p-2)(p-3)}\,\frac{1}{r_0}\,\frac{1}{r^{p-4}},\quad \left|\frac{d}{dr}(T(g))\right|\leq\frac{2^p}{p-2}\,\frac{1}{r_0}\,\frac{1}{r^{p-3}}
\end{equation} 
which shows, chosing $r_0$ large, that $T(g)\in B$. Similarly, using again that if $g,h\in B$, $|g(r)|\leq 2$ and $|h(r)|\leq 2$ for $r\geq r_0$, we get
\begin{multline}
 \label{B11}
\left|T(g)(r)-T(h)(r)\right|\leq \int_{r}^{+\infty} \int_{s}^{+\infty} \frac{1}{\sigma^{p-1}}2^{p-1}p|g(\sigma)-h(\sigma)|\,d\sigma\,ds\\
\leq C\int_{r}^{+\infty} \int_{s}^{+\infty} \frac{1}{\sigma^{2p-5}} \|g-h\|_{V}\,d\sigma\,ds\leq \frac{C}{r^{p-4}}\, \frac{1}{r_0^{p-3}} \|g-h\|_V
\end{multline} 
and
\begin{equation}
 \label{B12}
\left|\frac{d}{dr}\Big(T(g)(r)-T(h)(r)\Big)\right|\leq \frac{C}{r^{p-3}}\,\frac{1}{r_0^{p-3}}\|g-h\|_V,
\end{equation} 
which shows that $T$ is a contraction (chosing again $r_0$ large).

By fixed point, there exists a solution $G_1\in B$ of \eqref{B6'} for $r\geq r_0$. In particular
\begin{equation}
 \label{B13}
\sup_{r\geq r_0} \Big(\left|1-G_1(r)\right|r^{p-4}+\left|\frac{d G_1}{dr}(r)\right|r^{p-3}\Big)\leq 1.
\end{equation} 
Note that $Z_1=\frac{1}{r}G_1$ satisfies \eqref{B4} and \eqref{B5} (with $\ell=1$).

\EMPH{Step 2: extension of the solution} 

Let $(r_1,+\infty)$ be the maximal interval of existence of $G_1$, as a solution of the ODE $G_1''+\frac{1}{r^{p-1}}|G_1|^{p-1}G_1=0$. Multiplying this equation by $G_1'$ we obtain:
\begin{equation}
 \label{B14}
\frac{d}{dr}\left[\frac{1}{2}{G_1'}^2+\frac{1}{(p+1)r^{p-1}}|G_1|^{p+1}\right]=-\frac{p-1}{(p+1)r^p}|G_1|^{p+1}.
\end{equation} 
Letting $\Phi(r):=\frac{1}{2}{G_1'}^2+\frac{1}{(p+1)r^{p-1}}|G_1|^{p+1}$, we see that
\begin{equation}
 \label{B15}
\left|\Phi'(r)\right|\leq \frac{C}{r}\Phi(r)\text{ for }r\geq r_1.
\end{equation}
Thus if $r_1>0$, we get by Gronwall Lemma that $G_1(r)$ and $G_1'(r)$ remain bounded as $r\to r_1$, which is a contradiction with the definition of $r_1$ and the standard ODE blow-up criterion. As a conclusion, $r_1=0$, yielding a $C^2$ solution $Z_1=\frac{1}{r}G_1$ to:
\begin{equation}
 \label{B16}
\Delta Z_1+|Z_1|^{p-1}Z_1=0,\quad r>0.
\end{equation} 

\EMPH{Step 3: singularity at the origin}

In this step we show by contradiction that $Z_1=\frac{1}{r}G_1\notin L^{q_p}(\RR^3)$. Assume $Z_1\in L^{q_p}$. We know that \eqref{B16} is valid, in the distributional sense, on $\RR^3\setminus\{0\}$. We next prove that it holds in the distributional sense on the whole space $\RR^3$.

Let $\varphi\in C^{\infty}_0(\RR^3)$. Consider a radial cut-off function $\chi\in C^{\infty}_0(\RR^3)$ such that $\chi(x)=1$ for $|x|\leq 1$ and $\chi(x)=0$ for $|x|\geq 2$. We have
$$\int Z_1\Delta\varphi=\int Z_1(\Delta\varphi)\chi\left(\frac{x}{\eps}\right)+\int Z_1(\Delta\varphi)\left(1-\chi\left(\frac{x}{\eps}\right)\right)$$
and, by integration by parts (and using \eqref{B16}),
\begin{multline}
 \label{B17}
\int Z_1\Delta \varphi=\int Z_1\Delta\varphi\chi\left(\frac{x}{\eps}\right)+\frac{2}{\eps}\int Z_1\nabla \varphi\cdot\nabla \chi\left(\frac{x}{\eps}\right)\\
+\frac{1}{\eps^2} \int Z_1\Delta\chi\left(\frac{x}{\eps}\right)\varphi-\int |Z_1|^{p-1}Z_1 \left(1-\chi\left(\frac{x}{\eps}\right)\right)\varphi.
\end{multline} 
Let $\ovq_p:=\frac{3(p-1)}{3p-5}$ be the conjugate exponent to $q_p$. 
Then by H\"older inequality,
$$ \left|\frac{1}{\eps^2}\int Z_1\Delta\chi\left(\frac{x}{\eps}\right)\varphi\right|\leq \frac{C}{\eps^2}\left(\int_{|x|\leq 1} |Z_1|^{q_p}\right)^{\frac{1}{q_p}}\left(\int_{|x|\leq 1} \left(\Delta \chi\left(\frac{x}{\eps}\right)\right)^{\ovq_p}\right)^{\frac{1}{\ovq_p}}\leq C\eps^{\frac{3}{\ovq_p}-2}=C\eps^{\frac{p-3}{p-1}}\underset{\eps\to 0}{\longrightarrow} 0.$$
Using similar estimates, or dominated convergence, to treat the other terms in \eqref{B17}, we get as announced, letting $\eps\to 0$
\begin{equation}
 \label{B18}
\int Z_1\Delta \varphi=-\int |Z_1|^{p-1}Z_1\varphi.
\end{equation} 
A similar argument, testing equation \eqref{B16} against $Z_1\varphi^2$, shows that $\nabla Z_1\in L^2_{\loc}$.

Now, letting $H=|Z_1|^{p-1}\in L^{3/2}$ we see that $\Delta Z_1+HZ_1=0$, and thus by \cite{Trudinger68}, $Z_1\in L^{\infty}_{\loc}(\RR^3)$. By standard elliptic regularity, $Z_1\in C^{\infty}(\RR^3)$. This shows that $\partial_r Z_1(0)=0$. By \cite{JosephLundgren72}, we must have $Z_1(r)\approx\frac{1}{r^{\frac{2}{p-1}}}$ as $r\to\infty$, a contradiction. 
\end{proof}

\subsubsection{A Cauchy problem for finite energy solutions outside the origin}
\label{SSS:Cauchy}
If $I\subset \RR$ is an interval and $q\in [1,\infty]$, we let $L^q_I=L^q(\RR^3\times I)$. If $(v_0,v_1)\in \dot{H}^1(\RR^3)\times L^2(\RR^3)$, we denote by $S(t)(v_0,v_1)=v(t)$ the solution to the linear wave equation $\partial_t^2v-\Delta v=0$ with initial data $(v_0,v_1)$. Similarly, $\vec{S}(t)(v_0,v_1)=(v(t),\partial_t v(t))$.

Let us fix, once and for all, a function $\chi\in C^{\infty}(\RR^3)$, radial and such that:
\begin{equation}
 \label{B19}
\chi(r)=1\text{ if }r\geq \frac{1}{2},\quad \chi(r)=0\text{ if }r\leq \frac{1}{4}.
\end{equation} 
Denote, for $r_0>0$, $\chi_{r_0}(r)=\chi\left(\frac{r}{r_0}\right)$. 
\begin{lemma}
\label{L:CP}
Let $p>5$.
 There exists $\delta_0>0$ with the following property. Let $I$ be an interval with $0\in I$. Let $V\in L^{2(p-1)}_I$, radial in the variable $x$ and such that
\begin{equation}
 \label{B20}
\left\|D_x^{1/2}V\right\|_{L^4_I}<\delta_0\sqrt{r_0}^{\frac{p-5}{p-1}}\quad \text{and} \quad \|V\|_{L^{2(p-1)}_I}<\delta_0.
\end{equation} 
Consider $(h_0,h_1)\in (\hdot\times L^2)(\RR^3)$, radial such that
\begin{equation}
 \label{B21}
\left\|(h_0,h_1)\right\|_{\hdot\times L^2}<\delta_0 \sqrt{r_0}^{\frac{p-5}{p-1}}.
\end{equation} 
Then the Cauchy problem:
\begin{equation}
 \label{CPbis}
\left\{
\begin{gathered}
 \partial_t^2h-\Delta h=\left|V+\chi_{r_0}h\right|^{p-1}(V+\chi_{r_0}h)-|V|^{p-1}V\\
(h,\partial_t h)_{\restriction t=0}=(h_0,h_1)
\end{gathered}\right.
\end{equation}
is well-posed on the interval $I$. Furthermore, the corresponding solution $h$ satisfies:
\begin{equation}
 \label{B22}
\sup_{t\in I} \left\|\vec{h}(t)-\vec{S}(t)(h_0,h_1)\right\|_{\hdot\times L^2}\leq \frac{1}{100} \|(h_0,h_1)\|_{\hdot\times L^2}.
\end{equation}  
If $V=0$, one can always take $I=\RR$ and the preceding estimate can be improved to 
\begin{equation}
 \label{B23}
\sup_{t\in \RR} \left\|\vec{h}(t)-\vec{S}(t)(h_0,h_1)\right\|_{\hdot\times L^2}\leq\frac{C}{\sqrt{r_0}^{p-5}}\left\|(h_0,h_1)\right\|^p_{\hdot\times L^2},
\end{equation} 
for a constant $C>0$ depending only on $p$.
\end{lemma}
\begin{proof}
 By scaling considerations, we can assume $r_0=1$. The proof is close to the proof of Lemma 2.4 in \cite{DuKeMe12P} and the arguments of Section 2 of \cite{KeMe11} 
and we only sketch it. Recall the following version of Strauss' Lemma (see e.g. \cite[Lemma 3.2]{KeMe11} for a proof):
\begin{equation}
 \label{B24}
\forall f\in \hdot(\RR^3)\text{ radial},\quad  \forall r>0,\quad |f(r)|\leq \frac{C}{\sqrt{r}}\|f\|_{\hdot}.
\end{equation} 
For a small $\alpha>0$, let:
\begin{equation}
 \label{B25}
\BBB_{\alpha}=\left\{h\in L^8_I,\text{ radial s.t. }\tnorm{h}\leq\alpha \right\},
\end{equation} 
where
\begin{equation}
 \label{B26}
\tnorm{h}=\|h\|_{L^8_I}+\|D_x^{1/2}h\|_{L^{4}_I}+\|\nabla h\|_{L^{\infty}(I,L^2)}+\|\partial_t h\|_{L^{\infty}(I,L^2)}.
\end{equation} 
Note that by \eqref{B24},
\begin{equation}
 \label{B27}
\exists C>0,\quad \forall q\in [8,\infty], \quad \|\chi h\|_{L^q_I}\leq C\tnorm{h}.
\end{equation} 
Define, for $v\in \BBB_{\alpha}$, 
\begin{equation}
 \label{B28}
\Phi(v)(t)=S(t)(h_0,h_1)+\int_0^t \frac{\sin\left((t-s)\sqrt{-\Delta}\right)}{\sqrt{-\Delta}}F_V(v(s))\,ds,
\end{equation} 
where $F_V(v)=\left|V+\chi v\right|^{p-1}\left(V+\chi v\right)-|V|^{p-1}V$. We will show that if $\delta_0>0$ and $\alpha>0$ are small, $\Phi$ is a contraction on $\BBB_{\alpha}$. By Strichartz estimates, there exists $C_0>0$ such that
\begin{equation}
 \label{B29}
\tnorm{\Phi(v)}\leq C\left( \|(h_0,h_1)\|_{\hdot\times L^2}+\left\|D_x^{1/2}F_V(v)\right\|_{L^{4/3}_I}\right).
\end{equation} 
We have $F_V(v)=G(V+\chi v)-G(V)$, where $G(h)=|h|^{p-1}h$. By the chain rule for fractional derivatives (see \cite{KePoVe93}), 
\begin{multline*}
\left\|D_x^{1/2}(F_V(v))\right\|_{L^{4/3}_I}=\left\|D_x^{1/2}\left(G(V+\chi v)-G(V)\right)\right\|_{L^{4/3}_I}\\
\leq C\left(\left\|G'(V)\right\|_{L^2_I}+\left\|G'(V+\chi v)\right\|_{L^2_I}\right)\left\|D_x^{1/2}(\chi v)\right\|_{L^4_I}\\
+ C\left(\left\|G''(V)\right\|_{L^{\frac{2(p-1)}{p-2}}}+\left\|G''(V+\chi v)\right\|_{L^{\frac{2(p-1)}{p-2}}}\right)\left(\left\|D_x^{1/2}V\right\|_{L^4_I}+\left\|D_x^{1/2}(V+\chi v)\right\|_{L^4_I}\right)\left\|\chi v\right\|_{L^{2(p-1)}_I}\\
\leq C\tnorm{v}\left(\left\|V\right\|^{p-1}_{L^{2(p-1)}_I}+\left\|D_x^{1/2}V\right\|_{L^4_I}^{p-1}+\tnorm{v}^{p-1}\right).
\end{multline*}
Hence:
\begin{equation}
 \label{B30}
\left\|D_x^{1/2} \left(F_V(v)\right)\right\|_{L^{4/3}_I}
\leq C\alpha \left( \|V\|^{p-1}_{L^{2(p-1)}_I}+\left\|D_x^{1/2}V\right\|^{p-1}_{L^4_I}+\alpha^{p-1}\right)
\end{equation}
and, by \eqref{B29}, 
\begin{equation}
 \label{B31}
\tnorm{\Phi(v)}\leq C_0\left[\|(h_0,h_1)\|_{\hdot\times L^2}+
\alpha \left( \|V\|^{p-1}_{L^{2(p-1)}_I}+\left\|D_x^{1/2}V\right\|^{p-1}_{L^4_I}+\alpha^{p-1}\right)\right],
\end{equation} 
for some constant $C_0>0$.
Let $\alpha=2C_0\|(h_0,h_1)\|_{\hdot\times L^2}\leq 2C_0\delta_0$. Chosing $\delta_0\ll 1$, we see by \eqref{B31} that $v\in \BBB_{\alpha}$ implies $\Phi(v)\in \BBB_{\alpha}$. 

To get the contraction property, using again the chain rule for fractional derivatives, we get:
\begin{multline*}
\left\|F_V(V+\chi v)-F_V(V+\chi w)\right\|_{L^{4/3}_I}=\left\|D_x^{1/2}\left(G(V+\chi v)-G(V+\chi w)\right)\right\|_{L^{4/3}_I}\\
\leq 
C\left(\left\|G'(V+\chi v)\right\|_{L^2_I}+\left\|G'(V+\chi w)\right\|_{L^2_I}\right)\left\|D_x^{1/2}(\chi v-\chi w)\right\|_{L^4_I}\\
+ C\left(\left\|G''(V+\chi v)\right\|_{L^{\frac{2(p-1)}{p-2}}}+\left\|G''(V+\chi w)\right\|_{L^{\frac{2(p-1)}{p-2}}}\right)\\
\qquad \qquad\qquad\left(\left\|D_x^{1/2}(V+\chi v)\right\|_{L^4_I}+\left\|D_x^{1/2}(V+\chi w)\right\|_{L^4_I}\right)\left\|\chi (v-w)\right\|_{L^{2(p-1)}_I}\\
\leq C\tnorm{v-w}\left(\left\|V\right\|^{p-1}_{L^{2(p-1)}_I}+\left\|D_x^{1/2}V\right\|_{L^4_I}^{p-1}+\tnorm{v}^{p-1}+\tnorm{w}^{p-1}\right).
\end{multline*}

It remains to prove \eqref{B22} and \eqref{B23}. By \eqref{B28}, using that $\Phi(h(t))=h(t)$ we get, by Strichartz estimates and \eqref{B30},
$$ \tnorm{h-S(t)(h_0,h_1)}\leq C_0\alpha\left(\|V\|^{p-1}_{L^{2(p-1)}_I}+\left\|D_x^{1/2}V\right\|_{L^4_I}^{p-1}+\alpha^{p-1}\right),$$
which gives \eqref{B22} and \eqref{B23} since $\alpha=2C_0\|(h_0,h_1)\|_{\hdot\times L^2}$ and, if $V\neq 0$, $\|D_x^{1/2}V\|_{L^4_I}+\|V\|_{L^{2(p-1)}_I}\leq 2\delta_0$ can be chosen as small as necessary.
\end{proof}
\begin{remark}
 \label{R:potentiel}
\begin{enumerate}
 \item \label{I:V1}Let $V=\chi_{r_0}Z_1$, where $Z_1$ is given by Proposition \ref{P:singular}, and $r_0>0$. Then there exists $\theta_{r_0}>0$ such that $V$, $I_{r_0}=\left[-\theta_{r_0},\theta_{r_0}\right]$ satisfy the assumptions of Lemma \ref{L:CP}.
\item \label{I:V2}Let $R_0>1$ be a large parameter, and 
$$V(x,t)=\chi\left(\frac{x}{R_0+|t|}\right)Z_1(x).$$
Then $V$ satisfies the assumptions of Lemma \ref{L:CP} with $I=\RR$ and $r_0=R_0$.
\end{enumerate}
\end{remark}
The proof of \eqref{I:V1} and \eqref{I:V2} is very close to the proof of Claim 2.5 in \cite[Appendix A]{DuKeMe12P} and we omit it.
\subsubsection{Pointwise bounds for solutions with the compactness property}
\label{SSS:pointwise}
We recall here pointwise bounds on solutions of \eqref{CP} satisfying the assumptions of Proposition \ref{P:rigidityA}, which follow essentially from \cite{KeMe11}.
\begin{prop}
\label{P:pointwise1}
 Let $u$ be as in Proposition \ref{P:rigidityA} and assume that $u$ is global. Then there exists a constant $C>0$ such that
\begin{equation*}
\forall t\in \RR,\quad \forall R\geq 1,\quad \sqrt{R}|u(R,t)|+\int_{R}^{+\infty} |r \partial_r u(r,t)|^2\,dr+R^{p-3}\int_R^{+\infty} |r\partial_t u(r,t)|^2\,dr \leq \frac{C}{\sqrt{R}}. 
\end{equation*}
\end{prop}
(see \cite[Theorem 3.1, Corollary 3.7]{KeMe11}).
\begin{prop}
\label{P:pointwise2}
 Let $u$ be as in Proposition \ref{P:rigidityA} and assume that $T_+=+\infty$ and that $T_-$ is finite. Then there exists a constant $C>0$ such that
\begin{equation*}
\forall t\in (T_-,\infty),\quad \forall R\geq 1,\quad \sqrt{R}|u(R,t)|+\int_{R}^{+\infty} |r \partial_r u(r,t)|^2\,dr+R^{p-3}\int_R^{+\infty} |r\partial_t u(r,t)|^2\,dr \leq \frac{C}{\sqrt{R}}. 
\end{equation*}
\end{prop}
\begin{proof}[Sketch of proof]
By time translation, we can assume $T_-=0$.
As in \cite{KeMe11}, it is sufficient to show that $u$ satisfies the following bounds, which are the analogs of Theorem 3.1 of \cite{KeMe11}:
\begin{gather}
\label{pointwise_a}
\forall t\in (0,+\infty),\quad \forall r\geq 1,\quad |u(r,t)|\leq \frac{C}{r}\\
\label{pointwise_b}
\forall t\in (0,+\infty),\quad \forall R\geq 1,\quad \left(\int_{R}^{+\infty} |r \partial_r u(r,t)|^m\,dr\right)^{1/m}\leq \frac{C}{R^{1-a}}\\
\label{pointwise_c}
\forall t\in (0,+\infty),\quad \forall R\geq 1,\quad \left(\int_{R}^{+\infty} |r \partial_r u(r,t)|^2\,dr\right)^{1/m}\leq \frac{C}{R^{p(1-a)}},
 \end{gather}
where $a=\frac{2}{p-1}$, $m=\frac{1}{a}$. Note that 
\begin{equation}
\label{support_u}
\forall t>0,\qquad \supp u(\cdot,t)\subset B_{t}.
 \end{equation} 
The proofs of \eqref{pointwise_a}, \eqref{pointwise_b} and \eqref{pointwise_c} are the same as the corresponding proofs in Theorem 3.1 of \cite{KeMe11}, restricting to positive times. We refer the reader to this paper, highlighting the following small differences:
\begin{itemize}
 \item $g_1(r)$, $g_2(r)$, $g_3(r)$ should be defined as supremum over $t>0$ (instead of $t\in \RR$).  
\item In the proof of Lemma 3.4, at the end of page 1046, if $Mr\geq t_0$, the formula for $z_1(r+s,t_0)$ should be:
$$ z_1(r+s,t_0)=-\int_0^{r'} (r+s+\tau)|u|^{p-1}u(r+s+\tau,t_0-\tau)\,d\tau,$$
where (in view of \eqref{support_u}), we can take $r'$ defined by the equality $r+s+r'=t_0-r'$.
\item In the proof of Lemma 3.6, if $\frac{r_0}{2}\geq t_0$, replace (3.12) by
$$ r_0u_0(r_0,t_0)=-\frac{1}{2}\int_0^{t_0} \int_{r_0-(t_0-\tau)}^{r_0+(t_0-\tau)}\alpha|u|^{p-1}u\left(\alpha,\tau\right)\,d\alpha\,d\tau.$$
\end{itemize}
With these modifications, the proof works exactly the same as in the globally defined case and we omit it.
\end{proof}
\begin{corol}
\label{C:nochannel}
 Let $u$ satisfy the assumptions of Proposition \ref{P:rigidityA}, and $(T_-,T_+)$ its maximal interval of definition. Assume $0\in (T_-,T_+)$. Then 
\begin{equation}
\label{B33}
\forall t\in (T_-,T_+),\quad \vec{u}(t)\in \hdot\times L^2
\end{equation} 
and 
\begin{multline}
 \label{B34} \forall R>0,\quad \inf_{0<t<T_+} \int_{|x|\geq R+|t|} |\nabla u(x,t)|^2+(\partial_tu(x,t))^2\,dx\\
=\inf_{T_-<t<0} \int_{|x|\geq R+|t|} |\nabla u(x,t)|^2+(\partial_tu(x,t))^2\,dx=0.
\end{multline}
\end{corol}
\begin{proof}
 If $u$ is global, \eqref{B33} follows immediately from Proposition \ref{P:pointwise1}. If not, $\vec{u}(\cdot,t)$ is compactly supported and in $\Hcr$ for all $t$, and \eqref{B33} also holds.

Let us prove: 
$$ \inf_{0<t<T_+} \int_{|x|\geq R+|t|} |\nabla u(x,t)|^2+(\partial_tu(x,t))^2\,dx=0.$$
The proof of the other equality in \eqref{B34} is the same.
If $T_+=\infty$, this follows from Proposition \ref{P:pointwise1} or \ref{P:pointwise2}. If $T_+<\infty$, then
$$ \forall t\in (T_-,T_+), \quad \supp \vec u(\cdot,t)\subset \{|x|\leq T_+-t\},$$
which shows that $\int_{|x|\geq R+|t|} |\nabla u(x,t)|^2+(\partial_tu(x,t))^2\,dx=0$ if $t>0$ is such that $R+t\geq T_+-t$, concluding the proof.
\end{proof}

\subsection{Proof of the rigidity result}
\label{SS:proof}
This subsection is devoted to the proof of Proposition \ref{P:rigidityA}. The proof is an adaptation of the arguments of section 2 of our work \cite{DuKeMe12P}  on the energy-critical wave equation to the supercritical setting. Note that in the supercritical case, the singular stationary solution $Z_1$ given by Proposition \ref{P:singular} plays the role of the (regular) stationary solution $W$ of the energy-critical problem.

In all the subsection, we let $u$ be a solution of \eqref{CP} satisfying the assumptions of Proposition \ref{P:rigidityA}. We denote by $(T_-,T_+)$ its maximal interval of definition (which contains $0$), and:
$$ v(r,t)=ru(r,t), \quad v_0(r)=ru_0(r),\quad v_1(r)=ru_1(r).$$
By a straightforward integration by parts, if $r_0>0$ and $t\in (T_-,T_+)$, 
\begin{equation}
 \label{B35} \int_{r_0}^{+\infty} \left(\partial_r u(t,r)\right)^2r^2\,dr=\int_{r_0}^{+\infty}\left(\partial_r v(r,t)\right)^2\,dr+r_0u^2(r_0,t).
\end{equation} 
Note that the integrals in \eqref{B35} are finite thanks to Corollary \ref{C:nochannel}. We divide the proof into a few lemmas.
\begin{lemma}
 \label{L:B8}
There exist constants $\delta_1,C_1>0$ (independent of $u$ satisfying the assumptions of Proposition \ref{P:rigidityA}) such that, for any $r_0>0$, if
\begin{equation}
 \label{B36}
r_0^{-\frac{p-5}{p-1}}\int_{r_0}^{+\infty} \left((\partial_r u_0)^2+u_1^2\right)r^2\,dr=\delta\leq \delta_1,
\end{equation} 
then
\begin{equation}
 \label{B37}
\int_{r_0}^{+\infty} \left((\partial_rv_0)^2+v_1^2\right)\,dr\leq \frac{C_1}{r_0^{2p-5}}v_0^{2p}(r_0),
\end{equation} 
and, if $r$ and $r'$ satisfy $r_0\leq r<r'\leq 2r$, 
\begin{equation}
 \label{B38}
\left|v_0(r)-v_0(r')\right|\leq \frac{\sqrt{C_1}}{r^{p-3}}|v_0(r)|^p\leq \sqrt{C_1} |v_0(r)|\delta^{\frac{p-1}{2}}.
\end{equation} 
\end{lemma}
\begin{proof}
 We first note that \eqref{B38} is an easy consequence of \eqref{B37}. Indeed, assume that \eqref{B36} implies \eqref{B37}, and note that if \eqref{B36} holds for some $r_0$, it is still valid for any $r\geq r_0$. Then we have (using \eqref{B37} with $r$ instead of $r_0$):
\begin{multline*}
 |v_0(r)-v_0(r')|\leq \int_{r}^{r'} \left|\partial_r v_0(s)\right|\,ds \leq \sqrt{(r'-r)\int_{r}^{r'} |\partial_r v_0|^2}\\
\leq \sqrt{C_1}r^{3-p} |v_0(r)|^p\leq \sqrt{C_1} \delta^{\frac{p-1}{2}}|v_0(r)|.
\end{multline*}
The last inequality follows from the inequality $\frac{1}{r} (v_0(r))^2\leq \int_{r}^{+\infty}(\partial_ru_0)^2\rho^2\,d\rho$, consequence of \eqref{B35}, and from assumption \eqref{B36}.

We next prove \eqref{B37}. If $(f,g)\in \hdot(\RR^3)\times L^2(\RR^3)$ is radial, we denote by $\Psi_R(f,g)$ the following element $(\tf,\tg)$ of $\hdot\times L^2$:
\begin{equation}
 \label{B39}
\left\{
\begin{aligned}
 (\tf,\tg)(r)&=(f,g)(r)\text{ if }r\geq R\\
(\tf,\tg)(r)&=(f(R),0)\text{ if } r\leq R.
\end{aligned}\right.
\end{equation}
Let $u_{\lin}(r,t)=S(t)(u_0,u_1)$, $v_{\lin}=ru_{\lin}$. Define $(\tu_0,\tu_1)=\Psi_{r_0}(u_0,u_1)$, $\tu_{\lin}=S(t)(\tu_0,\tu_1)$, and $\tv_{\lin}=r\tu_{\lin}$. Finally, let $\tu$ be the solution of 
\begin{equation}
 \label{B40}
\left\{
\begin{gathered}
\partial_t^2\tu-\Delta \tu=\left|\chi_{r_0}\tu\right|^{p-1}\chi_{r_0}\tu\\
(\tu,\partial_t \tu)_{\restriction t=0} =(\tu_0,\tu_1).
\end{gathered}
\right.
\end{equation}   
By Lemma \ref{L:CP} (with $V=0$) if $\delta_1$ is small enough, $\tu$ is globally defined and (using the formula \eqref{B35}),
\begin{multline}
 \label{B41}
\left\|(\vec{\tu}(t)-\vec{\tu}_{\lin})(t)\right\|_{\Hen}\leq \frac{C}{\sqrt{r_0}^{p-5}}\left\|(\tu_0,\tu_1)\right\|^p_{\Hen}\\
\leq \frac{C}{\sqrt{r_0}^{p-5}}\left(\int_{r_0}^{+\infty} (\partial_r v_0)^2+v_1^2\,dr+\frac{1}{r_0}v_0^2(r_0)\right)^{\frac{p}{2}}.
\end{multline} 
Arguing exactly as in \cite[Proof of Lemma 2.8]{DuKeMe12P}, we deduce, using Lemma \ref{L:linear}, that the following holds for all $t\leq 0$ or for all $t\geq 0$:
\begin{multline}
 \label{B42}
\int_{r_0}^{+\infty} (\partial_rv_0)^2+v_1^2\,dr\\
\leq 2\int_{r_0+|t|}^{+\infty} (\partial_r\tv_{\lin}(r,t))^2+(\partial_t\tv_{\lin}(r,t))^2\,dr\leq 2\int_{r_0+|t|}^{+\infty} \left((\partial_r\tu_{\lin}(r,t))^2+(\partial_t\tu_{\lin}(r,t))^2\right)r^2dr\\
\leq C\int_{r_0+|t|}^{+\infty} \left((\partial_r\tu(r,t))^2+(\partial_t\tu(r,t))^2\right)r^2dr+\frac{C}{r_0^{p-5}} \left(\int_{r_0}^{+\infty} (\partial_r v_0)^2+v_1^2\,dr+\frac{1}{r_0}v_0^2(r_0)\right)^{p}.
\end{multline}
By finite speed of propagation, we deduce that \eqref{B42} holds for all $t\in (0,T_+)$ or for all $t\in (T_-,0)$, with $\tu$ replaced by $u$ in the last line. By Corollary \ref{C:nochannel}, 
$$\inf_{t\in [0,T_+)}\int_{r_0+|t|}^{+\infty} \left((\partial_ru(r,t))^2+(\partial_tu(r,t))^2\right)r^2dr
=\inf_{t\in (T_-,0]}\int_{r_0+|t|}^{+\infty} \left((\partial_ru(r,t))^2+(\partial_tu(r,t))^2\right)r^2dr=0
$$
and we get:
\begin{equation}
 \label{B43}
\int_{r_0}^{+\infty} (\partial_r v_0)^2+v_1^2\,dr \leq \frac{C}{r_0^{p-5}}\left(\int_{r_0}^{+\infty} (\partial_r v_0)^2+v_1^2\,dr+\frac{1}{r_0}v_0^2(r_0)\right)^p.
\end{equation} 
By formula \eqref{B35} and assumption \eqref{B36},
\begin{equation}
\label{B44}
 \frac{1}{r_0^{\frac{p-5}{p-1}}} \int_{r_0}^{+\infty} v_1^2+(\partial_r v_0)^2\,dr\leq \delta_1,
\end{equation} 
and it is easy to see that \eqref{B37} follows from \eqref{B43} and \eqref{B44} if $\delta_1>0$ is small enough, concluding the proof of the lemma.
\end{proof}
\begin{lemma}
\label{L:B9}
 There exists $\ell\in \RR$ such that
\begin{equation}
 \label{B45} \lim_{r\to \infty} v_0(r)=\ell.
\end{equation} 
Furthermore, there exists $C>0$ such that
\begin{equation}
\label{B46} 
\forall r\geq 1,\quad |v_0(r)-\ell|\leq \frac{C}{r^{p-3}}. 
\end{equation} 
\end{lemma}
\begin{proof}
 If $u$ is not global, then $v_0$ is compactly supported and the lemma is obvious. Assume that $u$ is global. Then by Proposition \ref{P:pointwise1}, $|v_0(r)|\leq C$ for some constant $C>0$ independent of $r>1$. Thus by \eqref{B38}, chosing $r_0\geq 1$ such that \eqref{B36} holds, we get, for $r\geq r_0$,
\begin{equation}
 \label{B47}
\left|v_0\left(2^{n+1}r\right)-v_0\left(2^{n}r\right)\right|\leq \frac{C}{\left(2^nr\right)^{p-3}}.
\end{equation} 
This shows that $\sum_{n\geq 0} \left|v_0(2^{n+1}r_0)-v_0(2^nr_0)\right|$ is finite, and thus the existence of
\begin{equation}
 \label{B48}
\ell =\lim_{n\to\infty} v_0(2^nr_0)\in \RR.
\end{equation}
By \eqref{B38},
\begin{equation}
 \label{B49} \lim_{r\to \infty} v_0(r)=\ell.
\end{equation}  
Finally, by \eqref{B47}, if $r\geq r_0$ we have
\begin{equation*}
\left|v_0(r)-\ell\right| \leq \sum_{j=0}^{+\infty} \left|v_0(2^jr)-v_0(2^{j+1}r)\right|\leq \frac{C}{r^{p-3}},
\end{equation*} 
which concludes the proof of Lemma \ref{L:B9}.
\end{proof}
\begin{lemma}
 \label{L:B10}
Let $\ell$ be as in Lemma \ref{L:B9} and assume $\ell=0$. Then $u=0$.
\end{lemma}
\begin{proof}
 \EMPH{Step 1} We first show that $(v_0,v_1)$ is compactly supported. Let $r_0>0$ such that \eqref{B36} is satisfied. Then by \eqref{B38} in Lemma \ref{L:B8},  chosing $\delta=r_0^{-\frac{p-5}{p-1}} \int_{r_0}^{+\infty} \left((\partial_ru_0)^2+u_1^2\right)r^2\,dr$ small, we get
\begin{equation*}
\forall n\geq 0,\quad |v_0(2^{n+1}r_0)|\geq \frac{3}{4} |v_0(2^nr_0)|,
\end{equation*} 
and thus by induction on $n$,
$$ \forall n\geq 0,\quad |v_0(2^nr_0)|\geq \left(\frac{3}{4}\right)^n|v_0(r_0)|.$$
On the other hand, by Lemma \ref{L:B9},
$$ |v_0(2^nr_0)|=|v_0(2^nr_0)-\ell|\leq \frac{C}{\left(2^nr_0\right)^{p-3}}.$$
Combining, we get that for all $n\geq 0$,
\begin{equation*}
 \frac{C}{\left(2^nr_0\right)^{p-3}}\geq \left(\frac 34\right)^n |v_0(r_0)|,
\end{equation*}
which shows that $v_0(r_0)=0$. By \eqref{B37} we deduce
$$ \int_{r_0}^{+\infty} (\partial_rv_0)^2+v_1^2\,dr=0,$$
concluding this step.

\EMPH{Step 2: end of the proof} 

Assume $(u_0,u_1)\neq (0,0)$ and let
$$\rho_0:=\inf\left\{ \rho>0\text{ s.t. } \int_{\rho}^{+\infty} \left((\partial_r u_0)^2+u_1^2\right)r^2\,dr=0\right\}\in (0,+\infty).$$
Let $\eps=\min\left(\frac{1}{2\sqrt{C_1}},\delta_1\right)$, where the constants $C_1$ and $\delta_1$ are given by Lemma \ref{L:B8}. Using the definition of $\rho_0$, the continuity of $v_0$ outside the origin and the continuity of the map $\rho\mapsto \rho^{-\frac{p-5}{p-1}}\int_{\rho}^{+\infty} \left((\partial_r u_0)^2+u_1^2\right)r^2\,dr$, we can chose $\rho_1\in (0,\rho_0)$ close to $\rho_0$ such that:
\begin{equation}
 \label{B52}
v_0(\rho_1)\neq 0\text{ and } \rho_1^{-\frac{p-5}{p-1}}\int_{\rho_1}^{+\infty} \left((\partial_r u_0)^2+u_1^2\right)r^2\,dr+\frac{|v_0(\rho_1)|^{p-1}}{\rho_1^{p-3}}<\eps.
\end{equation}
By Lemma \ref{L:B8}, 
\begin{equation}
 \label{B53}
|v_0(\rho_1)|=\left|v_0(\rho_0)-v_0(\rho_1)\right|\leq \frac{\sqrt{C_1}}{\rho_1^{p-3}}|v_0(\rho_1)|^p\leq \sqrt{C_1}\eps|v_0(\rho_1)|,
\end{equation}
a contradiction since $v_0(\rho_1)\neq 0$ and $\eps\sqrt{C_1}<1$. The proof is complete.  
\end{proof}
The following lemma completes the proof of Proposition \ref{P:rigidityA}:
\begin{lemma}
 \label{L:B11}
Let $\ell$ be as in Lemma \ref{L:B9}. Then $\ell=0$.
\end{lemma}
\begin{proof}
 We argue by contradiction. Assume $\ell\neq 0$. In particular $v_0$ is not compactly supported and thus $u$ is global.
Rescaling the solution and changing sign if necessary, we may assume $\ell=1$. Let $Z_1$ be the solution of $\Delta Z_1+|Z_1|^{p-1}Z_1=0$ on $\RR^3\setminus \{0\}$ given by Proposition \ref{P:singular}.

\EMPH{Step 1} Let $R_0>0$ be the large constant given by Remark \ref{R:potentiel}, \eqref{I:V2}. In this step we show:
\begin{equation}
 \label{B54}
\forall t\in \RR,\quad \supp\left(u(t)-Z_1,\partial_t u(t)\right)\subset B_{R_0},
\end{equation}
where $B_{R_0}=\left\{x\in \RR^3\text{ s.t. }|x|\leq R_0\right\}$. 

Assume without loss of generality that $t=0$. Let $h_0=u_0-Z_1$, $h_1=u_1$, $(H_0,H_1)=(rh_0,rh_1)$. Let $V(x,t)=\chi\left(\frac{x}{R_0+|t|}\right)Z_1(x)$ as in Remark \ref{R:potentiel}, \eqref{I:V2}. 

Let $r_1>R_0$ such that
\begin{equation}
 \label{B55}
\int_{r_1}^{+\infty} \left((\partial_rh_0)^2+h_1^2\right)r^2\,dr\leq \delta_0r_1^{\frac{p-5}{p-1}},
\end{equation}  
where $\delta_0$ is given by Lemma \ref{L:CP}. Let $(g_0,g_1)=\Psi_{r_1}(h_0,h_1)=\Psi_{r_1}(u_0-Z_1,u_1)$, where $\Psi_{R}$ is defined in \eqref{B39}. Let $g_{\lin}(t)=S(t)(g_0,g_1)$, and $g(t)$ be the solution of 
\begin{equation}
 \label{B56}
\begin{gathered}
 \partial_t^2g-\Delta g=\left|V+\chi_{R_0}g\right|^{p-1}(V+\chi_{R_0}g)-|V|^{p-1}V\\
\vec{g}_{\restriction t=0}=(g_0,g_1)
\end{gathered}
\end{equation}
given by Lemma \ref{L:CP}. We note that \eqref{B56} is exactly, for $|x|>R_0+|t|$, the equation  
$$ \partial_t^2 g-\Delta g=|Z_1+g|^{p-1}(Z_1+g)-|Z_1|^{p-1}Z_1$$
satisfied by $h=u-Z_1$. Since $(g,\partial_tg)_{\restriction t=0}=(u_0-Z_1,u_1)$ for $r>r_1$, we deduce by finite speed of propagation (see the comments after (2.27) in \cite{DuKeMe12P})
\begin{equation}
 \label{B57}
u(r,t)=Z_1(r)+g(r,t),\quad \partial_t u(r,t)=\partial_t g(r,t),\text{ for } t\in \RR,\quad r\geq r_1+|t|.
\end{equation}
By Lemma \ref{L:CP}, in view of \eqref{B55}, $g$ is globally defined and 
\begin{equation}
 \label{B58}
\sup_{t\in \RR} \left\|\vec{g}(t)-\vec{g}_{\lin}(t)\right\|_{\Hen}\leq \frac{1}{100}\|(g_0,g_1)\|_{\Hen}.
\end{equation} 
In view of Lemma \ref{L:linear}, the following holds for all $t\geq 0$ or for all $t\leq 0$
\begin{multline}
 \label{B59}
\int_{r_1}^{+\infty} \left((\partial_r H_0)^2+H_1^2\right)\,dr \leq 2\int_{r_1+|t|}^{+\infty} \left((\partial_r g_{\lin}(t))^2+(\partial_t g_{\lin}(t))^2\right)r^2\,dr\\
\leq 4\int_{r_1+|t|}^{+\infty} \left((\partial_rg(t))^2+(\partial_t g(t))^2\right)r^2\,dr+\frac{1}{100}\left[\int_{r_1}^{+\infty} (\partial_rH_0)^2+H_1^2\,dr+\frac{1}{r_1}H_0^2(r_1)\right].
\end{multline}
By \eqref{B57} and Corollary \ref{C:nochannel},
\begin{equation*}
 \lim_{t\to\pm \infty} \int_{r+|t|}^{+\infty} \left((\partial_rg(r,t))^2+(\partial_tg(r,t))^2\right)r^2\,dr=0.
\end{equation*}
Hence:
\begin{equation}
 \label{B60}
\int_{r_1}^{+\infty} \left((\partial_rH_0)^2+H_1^2\right)\,dr\leq \frac{1}{16r_1}H_0^2(r_1)
\end{equation}  
if \eqref{B55} is satisfied.

Fix $r_1$ such that \eqref{B55} holds. By the arguments leading to \eqref{B60},
$$\forall s\geq r_1,\quad \int_{s}^{+\infty} \left((\partial_rH_0)^2+H_1^2\right)\,dr\leq \frac{1}{16s} H_0^2(s).$$
Hence:
\begin{equation*}
 \left|H_0(2^{n+1}r_1)-H_0(2^nr_1)\right|\leq 2^{\frac{n}{2}}\sqrt{r_1}\sqrt{\int_{2^nr_1}^{2^{n+1}r_1}(\partial_r H_0)^2\,dr}\leq \frac{1}{4} H_0(2^nr_1).
\end{equation*}
This shows $\left|H_0(2^{n+1}r_1)\right|\geq \frac{3}{4} \left|H_0(2^nr_1)\right|$ and by induction,
$$ \left|H_0(2^nr_1)\right|\geq \left(\frac{3}{4}\right)^n \left|H_0(r_1)\right|.$$
Since by Lemma \ref{L:B9},
$$ |H_0(r)|=|ru_0(r)-rZ_1(r)|\leq |ru_0(r)-1|+|1-rZ_1(r)|\leq \frac{C}{r^2},$$
we get
$$ \forall n\geq 0,\quad \left(\frac{3}{4}\right)^n \left|H_0(r_1)\right|\leq \frac{C}{4^nr_1^2},$$
and thus $H_0(r_1)=0$. By \eqref{B60}, $\supp(H_0,H_1)\subset B_{r_1}$. This holds for any $r_1>R_0$ such that \eqref{B55} holds. As a consequence, the set $S=\{r_1>R_0\text{ s.t. } \supp(H_0,H_1)\subset B_{r_1}\}$ is nonempty and open. Since $S$ is also a closed subset of $(R_0,\infty)$, we get $S=(R_0,\infty)$ and thus as announced $\supp(H_0,H_1)\subset B_{R_0}$.

\EMPH{Step 2}
We show that $(u_0,u_1)=(Z_1,0)$ for $r>0$. Since $u_0\in L^{q_p}$ and, by Proposition \ref{P:singular}, $Z_1\notin L^{q_p}$ this will give the desired contradiction.

Let, for $t\in \RR$,
\begin{equation}
 \label{B62}
\rho(t):=\inf\left\{\rho>0\text{ s.t. } \int_{\rho}^{+\infty} \left((\partial_{r}h(r,t))^2+(\partial_t h(r,t))^2\right)r^2\,dr=0\right\}.
\end{equation} 
and
\begin{equation}
 \label{B63}
\rho_{\max}:=\sup_{t\in \RR} \rho(t),\quad r_0:=\frac{\rho_{\max}}{2}.
\end{equation} 
By Step 1, $\rho_{\max}\leq R_0$. We argue by contradiction, assuming $u_0\not \equiv Z_1$, and thus $\rho_{\max}>0$.

Let $V=\chi_{r_0}Z_1$, and $I_{r_0}=(-\theta_{r_0},\theta_{r_0})$ as in Remark \ref{R:potentiel}, \eqref{I:V1}. Taking a smaller $\theta_{r_0}$ if necessary, we can assume $\rho_{\max}-\frac{\theta_{r_0}}{2}>0$. Chose $t_0\in \RR$ such that $\rho(t_0)\geq \rho_{\max}-\frac{\theta_{r_0}}{2}$. Translating in time and taking a smaller $\theta_{r_0}$ if necessary, we can assume $t_0=0$ and
\begin{equation}
 \label{B64} 
\rho(0)\geq \rho_{\max}-\frac{\theta_{r_0}}{2}>r_0.
\end{equation} 
Recall from Step 1 the notations $h_0$, $h_1$, $H_0$, $H_1$, $g_0$, $g_1$, $g$. 

Chose $r_1\in (r_0,\rho(0))$ such that
$\frac{\rho(0)-r_1}{10r_1}\leq \frac{1}{2}$, $r_1+\theta_{r_0}>\rho_{\max}$, and 
\begin{equation}
 \label{B65}
0<\int_{r_1}^{+\infty} \left((\partial_r h_0)^2+h_1^2\right)r^2\,dr\leq \delta_0r_1^{\frac{p-5}{p-1}}.
\end{equation} 
Arguing as in Step 1 on the interval $I=(-\theta_{r_0},\theta_{r_0})$ instead of $\RR$, we get that the following holds for all $t\in [0,\theta_{r_0}]$ or for all $t\in [-\theta_{r_0},0]$ (see \eqref{B59}):
\begin{equation}
 \label{B66}
\int_{r_1}^{+\infty} \left((\partial_r H_0)^2+H_1^2\right)\,dr\leq 5\int_{r_1+|t|}^{+\infty} \left((\partial_r g(r,t))^2+\left(\partial_t g(r,t)\right)^2\right)r^2\,dr+\frac{1}{10r_1}H_0^2(r_1).
\end{equation} 
Since $r_1+\theta_{r_0}>\rho_{\max}$ we deduce
$$\int_{r_1+|t|}^{+\infty} \left((\partial_rg_0)^2+(g_1)^2\right)r^2\,dr=0\text{ for } t\in \pm \theta_{r_0}.$$
Hence:
\begin{multline*}
\int_{r_1}^{+\infty} (\partial_r H_0)^2+H_1^2\,dr\leq \frac{1}{10r_1} H_0^2(r_1)\leq \frac{1}{10r_1} \left(\int_{r_1}^{\rho(0)}|\partial_r H_0|\,dr\right)^2\\
\leq \frac{1}{10r_1}(\rho(0)-r_1)\int_{r_1}^{\rho(0)}(\partial_r H_0)^2\,dr.
\end{multline*}
Since $\frac{1}{10r_1}(\rho(0)-r_1)\leq \frac{1}{2}$ we deduce 
$$\int_{r_1}^{+\infty}(\partial_r H_0)^2+H_1^2\,dr=0$$
and thus, using the compact support of $H_0$, 
$$ \int_{r_1}^{+\infty}\left((\partial_r h_0)^2+h_1^2\right)\,r^2dr=0,$$
contradicting \eqref{B65}. The proof is complete.
\end{proof}

\section{Exclusion of self-similar, compact blow-up}
\label{S:self-sim}
In this Section we prove Proposition \ref{P:rigidityB}. We argue by contradiction, assuming that there exists a radial solution $u$ of \eqref{CP} with $T_+(u)<\infty$ and such that $K_+$ defined by \eqref{defK+'} is compact. We can assume without loss of generality that $T_+(u)=1$. 

We follow the lines of the proof of \cite[Section 6]{KeMe08}, with important simplifications given by the radiality assumption. 

We will use throughout the proof self-similar variables that we introduce now. Let $\delta\geq 0$ be a small parameter and
\begin{equation*}
 y=\frac{x}{1+\delta-t},\quad s=-\log(1+\delta-t).
\end{equation*}
Let 
\begin{equation*}
 w_{\delta}(y,s):=e^{-\frac{2s}{p-1}} u\left(e^{-s}y,1+\delta-e^{-s}\right), \quad y\in \RR^3,\quad s\in [0,-\log\delta).
\end{equation*}
If $\delta=0$, we will write 
\begin{equation*}
 w(y,s):=w_{0}(y,s)=e^{-\frac{2s}{p-1}} u\left(e^{-s}y,1-e^{-s}\right), \quad y\in \RR^3,\quad s\in [0,\infty).
\end{equation*}
Since $u$ satisfies \eqref{CP}, $w_{\delta}$ satisfies the following equation on $\RR^3\times [0,-\log\delta)$ ($\RR^3\times [0,\infty)$ if $\delta=0$):
\begin{equation}
 \label{C1'}
\partial_s^2w_{\delta} -\frac{1}{(1-r^2)^{\alpha}r^2} \partial_r\left((1-r^2)^{\alpha+1}r^2\partial_rw_{\delta}\right)+2r\partial_r\partial_sw_{\delta}+\frac{p+3}{p-1}\partial_sw_{\delta}+\frac{2(p+1)}{(p-1)^2}w_{\delta}-|w_{\delta}|^{p-1}w_{\delta}=0,
\end{equation} 
where $\alpha=\frac{2}{p-1}-1\in (-1,-1/2)$, and $r=|y|$.

By Lemma 4.15 of \cite{KeMe11}, 
$$ \forall t\in [0,1),\quad \supp \vec{u}(\cdot,t)\subset B_{1-t}.$$
Thus $K$ has compact closure in $H^{s_p}\times H^{s_p-1}$. Furthermore,
\begin{equation}
 \label{C2}
\forall s\in [0,-\log\delta),\quad \supp w_{\delta}\subset B_{1-\delta e^s},
\end{equation} 
and it is easy to check:
\begin{equation}
 \label{C3}
\forall s\in [0,-\log\delta),\quad (w_{\delta}(s),\partial_sw_{\delta}(s))\in H^{s_p}\times H^{s_p-1}.
\end{equation} 
For $\delta=0$, we have
$$w(y,s)=e^{-\frac{2s}{p-1}} u(e^{-s}y,1-e^{-s})=(1-t)^{\frac{2}{p-1}}u((1-t)y,t)$$
and thus (using also that $\supp w(s) \subset B_1$ for all $s\geq 0$),
\begin{equation}
 \label{C4}
K_0:=\{w(s),\; s\in [0,+\infty)\}\text{ has compact closure in }H^{s_p}.
\end{equation} 
We divide the proof into two lemmas.
\subsection{Energy estimates}
\begin{lemma}
 \label{L:C2}
Under the preceding assumptions,
\begin{equation}
 \label{C5}
\int_0^{+\infty} \int_0^1 (1-r^2)^{\alpha-1} \left(\partial_s w(r,s)\right)^2r^2\,dr<\infty.
\end{equation} 
\end{lemma}
\begin{proof}
For $\delta>0$,  $s\in [0,-\log \delta)$, we define:
\begin{multline}
 \label{C6} 
\widetilde{E}_{\delta}(s)=\frac{1}{2} \int_0^1 (1-r^2)^{\alpha+1} \left(\partial_r w_{\delta}\right)^2r^2\,dr+\frac{1}{2} \int_0^1 (1-r^2)^{\alpha} \left(\partial_s w_{\delta}\right)^2r^2\,dr\\+\frac{p+1}{p-1} \int_0^1 (1-r^2)^{\alpha} w_{\delta}^2\, r^2\,dr -\frac{1}{p+1}\int_0^1 (1-r^2)^{\alpha}|w_{\delta}|^{p+1}r^2\,dr.
\end{multline} 
 We note that since $(w_{\delta}(s),\partial_sw_{\delta}(s))\in H^{s_p}\times H^{s_p-1}$ for all $s\in [0,-\log\delta)$, and $\supp (w_{\delta},\partial_sw_{\delta})\subset B_{1-\delta}$, $\widetilde{E}_{\delta}(s)$ is well-defined (and finite). We will see in Step 3 that $\widetilde{E}_{\delta}(s)$ is a nondecreasing function of $s$.

\EMPH{Step 1} We show
\begin{equation}
 \label{C7}
\lim_{s\to -\log\delta} \widetilde{E}_{\delta}(s)=0.
\end{equation} 
It is sufficient to prove that each term in the definition \eqref{C6} of $\widetilde{E}_{\delta}(s)$ tends to $0$ as $s$ goes to $-\log\delta$. We will focus on $\int_0^1(1-r^2)^{\alpha} (\partial_sw_{\delta})^2r^2\,dr$, the proof for the other terms is similar and easier. We have
\begin{multline}
 \label{C8}
\partial_sw_{\delta}(y,s)=-\frac{2}{p-1}e^{-\frac{2s}{p-1}} u(e^{-s}y,1+\delta-e^{-s})\\
+e^{-\frac{2s}{p-1}-s} \partial_t u(e^{-s}y,1+\delta-e^{-s})-e^{-\frac{2s}{p-1}-s}r (\partial_r u)(e^{-s}y,1+\delta-e^{-s}).
\end{multline} 
Using that $1-r\geq \delta$ and $e^s\leq \frac{1}{\delta}$ if $s\in [0,-\log\delta)$ and $(e^{-s}y,1+\delta-e^{-s})\in \supp u$, we get that there exists $C_{\delta}>0$, depending only on $\delta>0$, such that
\begin{multline}
 \label{C9}
\int_0^1 (1-r^2)^{\alpha} (\partial_sw_{\delta})^2r^2\,dr\\
\leq C_{\delta}\int_0^1 \left[u^2(r,1+\delta-e^{-s})+(\partial_t u)^2(r,1+\delta-e^{-s})+(\partial_r u)^2(r, 1+\delta-e^{-s})\right]r^2\,dr.
\end{multline} 
Since $1+\delta-e^{-s}\to 1$ as $s\to -\log\delta$ and $K_+$ has compact closure in $H^{s_p}\times H^{s_p-1}\subset H^1\times L^2$ we deduce, using that $s_p>1$ (i.e. $\frac{2}{p-1}<\frac{1}{2}$) that the right-hand side of \eqref{C9} goes to zero as $s\to-\log\delta$, concluding this step.

\EMPH{Step 2} We show that there exists a constant $C_0>0$ independent of $\delta\in (0,1)$ such that
\begin{equation}
 \label{C10}
\forall \delta>0, \quad \widetilde{E}_{\delta}(0)\geq -C_0.
\end{equation} 
Indeed, the only nonnegative term in the definition of $\widetilde{E}_{\delta}(0)$ is 
\begin{multline*}
 -\int_0^1 (1-r^2)^{\alpha}|w_{\delta}|^{p+1}(r,0)r^2\,dr=-\int_0^1 (1-r^2)^{\alpha} |u|^{p+1}(r,\delta)r^2\,dr\\
=-\int_0^{\frac{1}{2}} (1-r^2)^{\alpha} |u|^{p+1}(r,\delta)r^2\,dr-\int_{\frac{1}{2}}^1 (1-r^2)^{\alpha} |u|^{p+1}(r,\delta)r^2\,dr.
\end{multline*}
Since $\alpha>-1$, $\int_{\frac{1}{2}}^1 (1-r^2)^{\alpha} |u|^{p+1}(r,\delta)r^2\,dr\leq C\|u(\cdot,\delta)\|^{p+1}_{L^{\infty}(1/2,1)}\leq C\|u(\cdot,\delta)\|_{H^{s_p}}^{p+1}$ by a one-dimensional Sobolev inequality.
Furthermore, by a critical three dimensional Sobolev inequality,
$$ \int_0^{\frac{1}{2}} (1-r^2)^{\alpha} |u|^{p+1}(r,\delta)r^2\,dr\leq C\|u(\cdot,\delta)\|_{L^{p+1}}^{p+1}\leq C \|u(\cdot,\delta)\|_{\dot{H}^{\frac{3(p-1)}{2(p+1)}}}^{p+1}\leq C\|u(\cdot,\delta)\|_{H^{s_p}}^{p+1}.$$
Since $u(\cdot,t)$ is bounded in $H^{s_p}$ for $t\in [0,1)$, \eqref{C10} follows.

\EMPH{Step 3} We show
\begin{equation}
 \label{C11} \int_0^{+\infty} (1-r^2)^{\alpha-1}(\partial_sw)^2(s,r)r^2\,dr\leq -\frac{C_0}{2\alpha},
\end{equation} 
where $C_0$ is the constant of Step 2. Indeed, multiplying equation \eqref{C1'} by $(1-r^2)^{\alpha}$ and integrating with respect to $r^2dr$ , we get
$$ \frac{d}{ds} \widetilde{E}_{\delta}(s)=-2\alpha\int_0^1 (1-r^2)^{\alpha-1}(\partial_sw_{\delta})^2(s,r)r^2\,dr.$$
Integrating between $0$ and $-\log\delta$ and using Steps 1 and 2, we deduce, for $\delta>0$,
\begin{equation}
 \label{C12} -2\alpha \int_0^{-\log\delta} \int_0^1 (1-r^2)^{\alpha-1} (\partial_s w_{\delta})^2(r,s)r^2\,dr\,ds\leq C_0,
\end{equation} 
i.e \eqref{C11} with $w$ replaced by $w_{\delta}$. Fix a large constant $A>0$ and let
\begin{equation}
 \label{C13}
B(\delta):=\int_0^A \int_0^{1-A^{-1}} (1-r^2)^{\alpha-1} (\partial_sw_{\delta})^2(r,s)r^2\,dr\,ds\leq -\frac{C_0}{2\alpha}.
\end{equation} 
The estimate \eqref{C11} will follow from the following convergence result:
\begin{equation}
 \label{C14}
\lim_{\delta\to 0} B(\delta)=B(0):=\int_0^A \int_0^{1-A^{-1}} (1-r^2)^{\alpha-1} (\partial_sw)^2(r,s)r^2\,dr\,ds.
\end{equation} 
Let us show \eqref{C14}. We have:
$$w_{\delta}(y,s)=\left(1-\delta e^s\right)^{-\frac{2}{p-1}} w\left(\frac{y}{1-\delta e^s},-\log(e^{-s}-\delta)\right)$$
and thus
\begin{align*}
 \partial_sw_{\delta}(y,s)&=\frac{2}{p-1}\delta e^s\left(1-\delta e^s\right)^{-\frac{2}{p-1}} w\left(\frac{y}{1-\delta e^s},-\log(e^{-s}-\delta)\right)\\
&\qquad + \left(1-\delta e^s\right)^{-\frac{2}{p-1}-1} \partial_sw\left(\frac{y}{1-\delta e^s},-\log(e^{-s}-\delta)\right)\\
&\qquad +\left(1-\delta e^s\right)^{-\frac{2}{p-1}-2} \delta e^s r\partial_r w\left(\frac{y}{1-\delta e^s},-\log(e^{-s}-\delta)\right)\\
&=W_1+W_2+W_3.
\end{align*}
Since $r$ and $s$ are bounded by $1-A^{-1}$ and $A$ respectively in the integrals defining $B(\delta)$, and $w(s)$, $r\partial_rw(s)$ are bounded in $L^2(r^2\,dr)$ uniformly with respect to $s\in [0,+\infty)$, it is easy to see that
$$\lim_{\delta \to 0} \int_0^A \int_0^{1-A^{-1}} (1-r^2)^{\alpha-1} \left(W_1^2+W_3^2\right)r^2\,dr\,ds=0.$$
Furthermore,
\begin{equation*}
 \lim_{\delta\to 0} \int_0^A \int_0^{1-A^{-1}}(1-r^2)^{\alpha-1}W_2^2\,r^2\,dr\,ds=\int_0^A \int_0^{1-A^{-1}} (1-r^2)^{\alpha-1}(\partial_s w(r,s))^2r^2\,dr\,ds,
\end{equation*}
by the change of variable $(\rho,\sigma)=\left(\frac{r}{1-\delta e^s},-\log(e^{-s}-\delta)\right)$. This proves as announced that $B(\delta)\to B(0)$ as $\delta\to 0$, concluding Step 3 and the proof of Lemma \ref{L:C2}.
\end{proof}
\subsection{Convergence to a stationary solution}
\begin{lemma}
\label{L:C3}
 There exists $w_*\in H^{s_p}\setminus\{0\}$, with $\supp w_*\subset B_1$ and a sequence $s_n\to +\infty$ such that $w(s_n)\to w_*$ in $H^{s_p}$ and
\begin{equation}
 \label{C15}
-\frac{1}{(1-r^2)^{\alpha}r^2}\frac{\partial}{\partial r} \left((1-r^2)^{\alpha+1} r^2\partial_r w_*\right)+\frac{2(p+1)}{(p-1)^2} w_*-|w_*|^{p-1}w_*=0.
\end{equation} 
\end{lemma}
\begin{proof}
 \EMPH{Step 1: convergence to $w_*$}
In this step we show that there exists $s_n\to+\infty$ and $w_*\in H^{s_p}$ such that
\begin{equation}
 \label{C16}
\forall T\geq 0,\quad \lim_{n\to\infty} \left\| w(\cdot,s_n+T)-w_*\right\|_{H^{s_p}}=0.
\end{equation} 
Indeed, using that the sequence $\left\{w(\cdot,n)\right\}_n$ stays in a compact subset of $H^{s_p}$, we get a subsequence $\left\{w(\cdot,s_n)\right\}_n$ and an element $w_*$ of $H^{s_p}$ such that
\begin{equation}
 \label{C17}
\lim_{n\to\infty} \left\| w(\cdot,s_n)-w_*\right\|_{H^{s_p}}=0.
\end{equation} 
Furthermore, by Cauchy-Schwarz inequality,
\begin{multline*}
 \int |w(y,s_n)-w(y,s_n+T)|^2\,dy\leq \int \left|\int_{s_n}^{s_n+T} \partial_s w(y,s)\,ds\right|^2\,dy\leq T\int_{s_n}^{s_n+T} \int |\partial_sw(y,s)|^2\,dy\,ds.
\end{multline*}
By Lemma \ref{L:C2}, the right-hand side of the preceding inequality goes to $0$ as $n$ goes to infinity. In view of \eqref{C17}, we get
$$\lim_{n\to\infty} \|w(\cdot,s_n+T)-w^*\|_{L^2}=0,$$
and \eqref{C16} follows by compactness in $H^{s_p}$. 

\EMPH{Step 2: elliptic equation} We show that $w_*$ satisfies \eqref{C15}. 

Let 
$$ \left(v_{0n}(y),v_{1n}(y)\right)=\left(e^{-\frac{2s_n}{p-1}} u\left(e^{-s_n}y,1-e^{-s_n}\right),e^{-\frac{2s_n}{p-1}-s_n} \partial_t u\left(e^{-s_n}y,1-e^{-s_n}\right)\right),$$
and $v_n$ be the solution of \eqref{CP} with initial data $(v_{0n},v_{1n})$. Extracting a subsequence if necessary, we can assume
$$ \lim_{n\to \infty} (v_{0n},v_{1n})=(v_0,v_1) \text{ in }\Hcr,$$
where $(v_0,v_1)\in \Hcr$. We let $v$ be the solution of \eqref{CP} with initial data $(v_0,v_1)$. Let $s\in [0,T_+(v))$. By standard perturbation theory (see Theorem 2.11 of \cite{KeMe11}), for large $n$, $s$ is in the interval of existence of $v_n$ and
\begin{equation}
 \label{C18}
\lim_{n\to\infty} \left\|v_n(\cdot,s)-v(\cdot,s)\right\|_{\dot{H}^{s_p}}=0.
\end{equation} 
Furthermore, if $s<1$ and $T$ is such that $s=1-e^{-T}$, we have
$$ v_n(y,s)=e^{-\frac{2 s_n}{p-1}} u\left(e^{-s_n}y,1-e^{-s_n}+e^{-s_n}s\right)$$
and thus
\begin{equation}
 \label{C19}
w(y,s_n+T)=e^{-\frac{2s_n}{p-1}}e^{-\frac{2T}{p-1}}u\left(e^{-s_n-T}y,1-e^{-s_n-T}\right)=e^{-\frac{2T}{p-1}}v_n\left(e^{-T}y,1-e^{-T}\right).
\end{equation} 
By Step 1, $w(\cdot,s_n+T)\to w_*$ as $n\to\infty$ in $H^{s_p}$. Letting $n\to\infty$ in \eqref{C19}, we get that for small $T\geq 0$,
\begin{equation}
 \label{C20}
w_*(y)=e^{-\frac{2T}{p-1}} v\left(e^{-T}y,1-e^{-T}\right).
\end{equation} 
Since $v$ satisfies \eqref{CP}, we see that $w^*$ is as solution of \eqref{C1'} which is independent of $s$, which gives exactly the announced elliptic equation \eqref{C15}.

\EMPH{Step 3: further properties of $w_*$} 

Obviously, since supp $w(\cdot,s)\subset B_1$ for all $s\geq 0$, we have $\supp w_*\subset B_1$. It remains to prove that $w_*$ is not identically $0$. By \eqref{C20}, if $w_*\equiv 0$, then $v\equiv 0$. As a consequence,
$$ \lim_{n\to\infty}\left\| (v_{0n},v_{1n})\right\|_{\Hcr} =0.$$
By the definition of $(v_{0n},v_{1n})$, we get
$$\lim_{n\to\infty} \left\| \vec{u}(\cdot, 1-e^{-s_n})\right\|_{\Hcr}=0$$
which, by the small data theory for \eqref{CP}, shows that $u\equiv 0$, contradicting our assumptions.
\end{proof}
\subsection{End of the proof}
Let us denote by $w_*'$ the radial derivative $\partial_rw_*$ of $w_*$.

Since $w_*\in H^{s_p}$ and $w_*$ is radial, we deduce that $w_*$ is continuous outside $r=0$. Furthermore, using that $w_*(r)=0$ for $r\geq 1$, we obtain, for $\frac{1}{2}\leq r\leq 1$,
$$|w_*(r)|=\left|\int_r^1 w_*'(\rho)\,d\rho\right|\leq (1-r^{1/2})\sqrt{\int_r^1 (w_*')^2\,d\rho}\leq C(1-r)^{1/2},$$
where the positive constant $C$ is independent of $r\in [1/2,1]$. By \eqref{C15},
$$ \left|\frac{d}{dr}\left((1-r^2)^{\alpha+1} r^2 w_*'\right)\right|\leq C(1-r)^{\frac{1}{2}+\alpha}.$$
Since $\alpha>-1$ we get, integrating between $r\in [1/2,1)$ and $1$, that $w_*'$ is continuous for $r\geq 1/2$ and that there exists $\ell_*\in \RR$ such that
\begin{equation}
 \label{C21} \left|(1-r^2)^{\alpha+1} r^2w_*'(r)-\ell_*\right|\leq C(1-r)^{\frac{3}{2}+\alpha},\quad \frac 12\leq r< 1.
\end{equation} 
If $\ell_*\neq 0$, then $r^2w_*'(r)\sim \frac{\ell_*}{(1-r^2)^{\alpha+1}}=\frac{\ell_*}{(1-r^2)^{\frac{2}{p-1}}}$ as $r\to 1$. This contradicts the fact that $w_*'\in L^{\frac{p-1}{2}}$ (which follows from $w_*\in \dot{H}^{s_p}$ and Sobolev embedding). 

Thus $\ell_*=0$ and by \eqref{C21},
\begin{equation}
 \label{C22}
\left|w_*'(r)\right|\leq C(1-r)^{1/2},\quad \frac{1}{2}\leq r\leq 1.
\end{equation} 
By \eqref{C15}
\begin{multline*}
 (1-r)^{\alpha+1}\left|w_*'\right|\leq C\int_r^1 (1-\rho)^{\alpha} |w_*(\rho)|\,d\rho\leq C\int_r^1 (1-\rho)^{\alpha} \int_{\rho}^1 \left|w_*'(\sigma)\right|\,d\sigma\,d\rho\\
\leq C(1-r)^{\alpha+1}\int_r^1 \left|w_*'(\sigma)\right|\,d\sigma.
\end{multline*}
Hence $\left|w_*'(r)\right|\leq C\int_r^1 \left|w_*'(\sigma)\right|\,d\sigma$. Since $w_*'$ is continuous on $[1/2,1]$ and $w_*(1)=w_*'(1)=0$, we deduce that $w_*=0$ close to $r=1$, and thus $w_*\equiv 0$ by standard unique continuation. This is a contradiction with Lemma \ref{L:C3}, concluding the proof of Proposition \ref{P:rigidityB}.

\bibliographystyle{alpha} 
\bibliography{toto}
\end{document}